\documentclass[12pt]{amsart}
\usepackage{amsmath,amssymb}

\DeclareMathOperator{\End}{End}
\DeclareMathOperator{\Gal}{Gal}
\DeclareMathOperator{\rk}{rk}
\DeclareMathOperator{\Supp}{Supp}
\DeclareMathOperator{\Aut}{Aut}
\DeclareMathOperator{\ord}{ord}

\begin{document}
\newtheorem{thm}{Theorem}[section]
\newtheorem{lem}[thm]{Lemma}
\newtheorem{dfn}[thm]{Definition}
\newtheorem{cor}[thm]{Corollary}
\newtheorem{conj}[thm]{Conjecture}
\newtheorem{clm}[thm]{Claim}
\theoremstyle{remark}
\newtheorem{exm}[thm]{Example}
\newtheorem{rem}[thm]{Remark}
\def\N{{\mathbb N}}
\def\G{{\mathbb G}}
\def\Q{{\mathbb Q}}
\def\R{{\mathbb R}}
\def\C{{\mathbb C}}
\def\P{{\mathbb P}}
\def\Z{{\mathbb Z}}
\def\v{{\mathbf v}}
\def\x{{\mathbf x}}
\def\O{{\mathcal O}}
\def\M{{\mathcal M}}
\def\kbar{{\bar{k}}}
\def\tr{\mbox{Tr}}
\def\id{\mbox{id}}
\def\qed{{\tiny $\clubsuit$ \normalsize}}

\renewcommand{\theenumi}{\alph{enumi}}

\title{Ideals of degree one contribute most of the height}

\author{Aaron Levin}
\author{David McKinnon}

\begin{abstract}
Let $k$ be a number field, $f(x)\in k[x]$ a polynomial over $k$ with $f(0)\neq 0$, and $\O_{k,S}^*$ the group of $S$-units of $k$, where $S$ is an appropriate finite set of places of $k$.
In this note, we prove that outside of some natural exceptional set $T\subset \O_{k,S}^*$, the prime ideals of $\O_k$ dividing
$f(u)$, $u\in \O_{k,S}^*\setminus T$, mostly have degree one over $\Q$; that
is, the corresponding residue fields have degree one over the prime
field.  We also formulate a conjectural analogue of this result for
rational points on an elliptic curve over a number field, and
deduce our conjecture from Vojta's Conjecture.  We prove this
conjectural analogue in certain cases when the elliptic curve has
complex multiplication.
\end{abstract}

\maketitle

\vspace{-5pt}
\section{Introduction}

If $a$ is an algebraic integer in a number field $k$ and $f(x)\in \O_k[x]$ a polynomial,
then the ideals dividing $f(a)$ are simply the ideals $I$ such that
$f(a)\equiv 0\pmod I$.  Heuristically, the larger the cardinality of the
residue ring $\O_k/I$, the smaller the probability that $f(a)$ and
$0$ are the same.

The purpose of this paper is to make this notion precise, to
generalise it, and to prove it in the case described above.  More
specifically, in Theorem~\ref{main}, using a result of Corvaja and
Zannier, we prove a precise version of this notion for $\mathbb{G}_m$,
and in Theorem~\ref{ellmain}, we state a conjectural analogue of
Theorem~\ref{main} for elliptic curves over a number field, and show
that it is a consequence of Vojta's Conjecture (\cite{Vo}).

A theorem of the second author proves Vojta's Conjecture in a relevant
special case, and we deduce an unconditional version of
Theorem~\ref{ellmain} in that case.  Specifically, if the elliptic
curve $E/k$ has complex multiplication, and if the algebraic point $P$
is defined over the compositum of $k$ with $\mbox{End}(E)\otimes\Q$,
then we can deduce Theorem~\ref{ellmain} without the hypothesis that
Vojta's Conjectures are true.

\section{Main Theorem}

Let $f(x)\in k[x]$ be a polynomial over a number field $k$ and $p\in \mathbb{Z}$ a prime.  The heuristic mentioned in the introduction suggests that a prime $\mathfrak{p}$ of $k$ lying above $p$ is likelier to divide $f(a)$ for $a\in k$ if the residue field $\O_k/\mathfrak{p}$ is small.  Our main theorem will give one possible precise interpretation of this notion, where we view $\O_k/\mathfrak{p}$ as being small if $\O_k/\mathfrak{p}$ has degree one over its prime field.  There is, however, an obvious way that our heuristic can fail.  Suppose, for example, that $f$ and $a$, and hence $f(a)$, are actually defined over a proper subfield $k'$ of $k$.  Then the size of $\O_{k'}/(\mathfrak{p}\cap \O_{k'})$, and not $\O_k/\mathfrak{p}$, is clearly the relevant quantity.  In the simplest case, when $k/\mathbb{Q}$ is Galois and $f$ is irreducible, our main theorem says, in essence, that for $S$-units $u$ of $k$ this is in fact the only way our heuristic can fail, i.e., $f(u)$ is ``mostly'' supported on primes of $k$ of degree one over $\mathbb{Q}$ unless $f(u)$ is rational, in an appropriate sense, over a proper subfield of $k$.

The statement of the main theorem requires a fair amount of notation.
We summarize this notation as follows:\\

\begin{tabular}{ll}
$k$ & Extension of $\Q$ of degree $d\neq 1$ \\
$L$ & Galois closure of $k$ over $\mathbb{Q}$\\
$\Gal(L/\Q)$ & The Galois group of $L$ over $\mathbb{Q}$\\
$\O_k$ & Ring of integers of $k$ \\
$f(x)$ & Nonconstant polynomial in $\O_k[x]$ with $f(0)\neq 0$ \\
$f_1,\ldots,f_N$ & The monic irreducible factors of $f$ over $L$\\
$S$ & Finite set of places of $k$ containing the archimedean places\\ & such that if $v\in S$ and $v$ and $v'$ lie above the same rational\\ & prime $p\in \Z$ then $v'\in S$.\\
$\O_{k,S}$ & Ring of $S$-integers of $k$ \\
$\O_{k,S}^*$ & Group of $S$-units of $k$\\
$\tau$ & The involution $\tau(u)=u^{-1}$ of $\O_{k,S}^*$.\\
$\O_{k,S}^{*\phi}$ & For a homomorphism $\phi$, the subgroup of $\O_{k,S}^*$ consisting of \\& elements $u$ such that $\phi(u)=u$.\\
$I(f(u))$ & $f(u)\O_{k,S}$, for $u\in \O_{k,S}^*$; the smallest ideal of $\O_{k,S}$ such that \\ 
 & $f(u)\equiv 0\pmod{I(f(u))}$ \\ 
$J(f(u))$ & \parbox[t]{6in}{Smallest ideal dividing $I(f(u))$ such that for every prime \\
$\mathfrak{p}$ dividing $J(f(u))$, $\O_{k,S}/\mathfrak{p}$ has degree greater than one \\ over the prime
field} \\
$N(I)$ & The norm of $I$ over $\Q$, for any ideal $I$ of $\O_k$ or $\O_{k,S}$\\
$H_k(x)$ & The relative multiplicative Weil height of $x\in k$\\
$H(x)$ & The absolute multiplicative Weil height of $x$,\\
&equal to $H_k(x)^{1/d}$ for $x\in k$\\
$h(x)$ & The absolute logarithmic Weil height of $x$, equal to $\log H(x)$
\end{tabular}

\vspace{.1in}

We can now state the main theorem:

\begin{thm}\label{main}
Let $\epsilon>0$.  Let $f(x)\in \O_k[x]$ satisfy $f(0)\neq 0$.  Then there exists a finite set of places $S'$ of $L$ such that for every $u\in \O_{k,S}^*$ either
\begin{enumerate}
\item
\label{maineqn}
\begin{equation*}
N(J(f(u)))<H(u)^{\epsilon}
\end{equation*}
or\\
\item
for some $i$,
\begin{equation}
\label{rat}
f_i(u)\O_{L,S'}=\alpha \O_{L,S'}
\end{equation}
for some $\alpha$ that lies in a proper subfield of $L$ not containing $k$ (in particular, if $k/\Q$ is Galois, $\alpha$ lies in a proper subfield of $k$).
\end{enumerate}

With the exception of finitely many elements, the set of elements in $\O_{k,S}^*$ not satisfying \eqref{maineqn} is contained in a finite union of cosets in $\O_{k,S}^*$ of the form
\begin{equation*}
T=u_1\O_{k,S}^{*\sigma_1}\cup\cdots\cup u_{m'}\O_{k,S}^{*\sigma_{m'}}\cup u_{m'+1}\O_{k,S}^{*\sigma_{m'+1}\tau}\cup\cdots\cup u_{m}\O_{k,S}^{*\sigma_{m}\tau},
\end{equation*}
where $u_1,\ldots, u_m\in \O_{k,S}^*$ and $\sigma_1,\ldots,\sigma_{m}\in \Gal(L/\Q)\setminus \Gal(L/k)$ (not necessarily distinct) are effectively computable.
\end{thm}

An alternative formulation of Theorem \ref{main} involving only heights is given in Corollary \ref{heightmain}.

Note that $H(f(u))\ll H(u)^{\deg f}$ and that
\begin{equation*}
H_k(f(u))=C_uN(I(f(u)))=C_uN(J(f(u)))N(I(f(u))/J(f(u))), 
\end{equation*}
where $C_u$ is a real number (roughly equal to the archimedean part of the height of
$f(u)^{-1}$) satisfying $C_u\ll H(u)^\epsilon$ (see Lemma \ref{Roth}).  Thus, Theorem~\ref{main}
implies that for $u\in \O_{k,S}^*\setminus T$, $f(u)$ is ``mostly'' supported on primes of $k$ of degree one
over $\Q$.

We mention also that the group $\O_{k,S}^{*\sigma_i}$ is the same as $\O_{F,S_F}^*$, where $F$ is the fixed field of $\sigma_i$ and $S_F$ is the set of places of $F$ lying below places of $S$.
\vspace{.1in}

Before we begin the proof, we introduce some
notation.  For a number field $k$ we denote the set of inequivalent places of $k$ by $M_k$.
We define the function $\log^-$ for positive real numbers $x$ by
$\log^-(x)=\min\{0,\log(x)\}$.  For a place $v\in M_k$, we normalize
the corresponding absolute value $|\cdot|_v$ in such a way that the
product formula holds and $H(x)=\prod_{v\in M_k}\max\{1,|x|_v\}$.\\

\noindent
{\it Proof:} \/ Consider the set
\begin{equation*}
U=\{u\in \O_{k,S}^*\mid N(J(f(u)))\geq H(u)^\epsilon\}.
\end{equation*}
Let $L$ be a Galois closure of $k$ over $\Q$.  Let $\mathfrak{p}$ be a prime of $\O_k$ of inertia degree greater than one over $\mathbb{Q}$, lying above a rational prime $p\in \Z$.  Let $\mathfrak{q}$ be a prime of $\O_L$ lying above $\mathfrak{p}$.  Then $\mathfrak{q}$ again has inertia degree greater than one over $\mathbb{Q}$.  Let $D=D(\mathfrak{q}/p)\subset \Gal(L/\mathbb{Q})$ be the decomposition group of $\mathfrak{q}$ and let $L^D$ be the decomposition field.  Then $k\not\subset L^D$ since $\mathfrak{p}$ has inertia degree greater than one.  It follows that there exists $\sigma\in \Gal(L/\mathbb{Q})$ such that $\sigma(\mathfrak{q})=\mathfrak{q}$, $\sigma\not\in \Gal(L/k)$.  

Let $S_L$ be the set of places of $L$ lying above places of $S$.  Let $J'(f(u))=J(f(u))\O_{L,S_L}$.  Let $\mathfrak{q}$ be a prime of $\O_{L,S_L}$ dividing $J'(f(u))$.  From the above discussion and the definition of $J(f(u))$, there exists an element $\sigma\in \Gal(L/\mathbb{Q})\setminus \Gal(L/k)$ such that $\sigma(\mathfrak{q})=\mathfrak{q}$.  Let $\Gal(L/\mathbb{Q})\setminus \Gal(L/k)=\{\sigma_1,\ldots,\sigma_m\}$.  For $i=1,\ldots, m$, define the ideal $J_i'(f(u))$ to be the smallest ideal of $\O_{L,S_L}$ dividing $J'(f(u))$ such that
$\sigma_i(J_i'(f(u)))=J_i'(f(u))$.  Then $J'(f(u))$ divides $J_1'(f(u))\cdots J_m'(f(u))$.  Note also that $N(J'(f(u)))\geq N(J(f(u)))$.  Let
\begin{equation*}
U_i=\{u\in U\mid N(J_i'(f(u))) \geq H(u)^{\epsilon/m}\}.
\end{equation*}
Then clearly $U\subset \cup_{i=1}^m U_i$.

Let $r\in \{1,\ldots, m\}$.  By definition, $J_r'(f(u))$ divides both $f(u)\O_{L,S_L}$ and $f^{\sigma_r}(\sigma_r(u))\O_{L,S_L}$ for
all $u$ (where $f^{\sigma_r}$ denotes the image of $f$ under the natural action of $\sigma_r$).  For $u\in U_r$, we therefore obtain:
\begin{align*}
[L:\mathbb{Q}]\Sigma_{v\in M_L} \log^-\max\{|f(u)|_v,|f^{\sigma_r}(\sigma_r(u))|_v\}&\leq - \log N(J_r'(f(u))) \\
&\leq -\log H(u)^{\epsilon/m}\\
&\leq -\frac{\epsilon}{m} h(u).
\end{align*}

Theorem~\ref{main} will follow essentially from Proposition~4 of \cite{CZ}:

\begin{lem}[Corvaja, Zannier]\label{vojta}
Let $f(x),g(x)\in L[x]$ be polynomials that do not
vanish at $x=0$.  Then for every $\epsilon>0$, all but finitely many
solutions $(u,u')\in(\O_{L,S_L}^*)^2$ to the inequality
\[\sum_{v\in M_L} \log^-\max\{|f(u)|_v,|g(u')|_v\} < -\epsilon 
(\max\{h(u),h(u')\})\] 
are contained in finitely many effectively computable translates of one-dimensional subgroups
of $\G_m^2$.
\end{lem}

Since $h(u)=h(\sigma_r(u))$ and $u,\sigma_r(u)\in \O_{L,S_L}^*$, taking $g=f^{\sigma_r}$ it follows immediately from
Lemma~\ref{vojta} that all but finitely many elements of the set $V_r=\{(u,\sigma_r(u))\mid u\in U_r\}$ are contained in finitely many effectively computable translates of one-dimensional subgroups of $\G_m^2$.  Let $X$ be a translate of a one-dimensional subgroup of $\G_m^2$ that contains infinitely many elements of $V_r$.  Let $(v,\sigma_r(v))\in X\cap V_r$.  Taking $u=v'/v\in \O_{k,S}^*$, where $(v',\sigma_r(v'))\in X\cap V_r$, we see that infinitely many elements of the form $(u,\sigma_r(u))$, $u\in \O_{k,S}^*$, will lie in the associated one-dimensional subgroup in $\G_m^2$.  We now classify the possibilities for such a one-dimensional subgroup.

Suppose there exists $a,b\in \mathbb{Z}$, not both zero, such that 
\begin{equation}
\label{depeq}
u^a\sigma_r(u)^b=1,
\end{equation}
for infinitely many $u\in \O_{k,S}^*$.  We claim that $a=\pm b$.  Let $l$ be the order of $\sigma_r$.  Then
\begin{equation*}
u^{b^l}=\sigma_r^l(u)^{b^l}=\sigma_r^{l-1}(u)^{-ab^{l-1}}=\cdots=u^{(-a)^l}.
\end{equation*}
So $u^{b^l-(-a)^l}=1$ for infinitely many $u\in \O_{k,S}^*$.  This implies that $b^l=(-a)^l$, or $a=\pm b$, as claimed.

Suppose first that $a=-b$.  Then for any $u\in \O_{k,S}^*$ satisfying \eqref{depeq} we have $\sigma_r(u^a)=u^a$.  So $u^a\in \O_{k,S}^{*\sigma_r}=F\cap \O_{k,S}^*$, where $F$ is the fixed field of $\sigma_r$.  It follows that $\O_{k,S}^{*\sigma_r}$ has finite index in $\{u\in \O_{k,S}^* \mid u^a\sigma_r(u)^{-a}=1\}$ and that $\{u\in \O_{k,S}^*\mid (u,\sigma_r(u))\in X\cap V_r\}$ is contained in a finite number of cosets of $\O_{k,S}^{*\sigma_r}$ in $\O_{k,S}^*$.

Suppose now that $a=b$.  Then for any $u\in \O_{k,S}^*$ satisfying \eqref{depeq} we have $\sigma_r(u^a)=u^{-a}$.  By definition, we have $u^{-a}\in \O_{k,S}^{*\sigma_r\tau}$.  Then, as above, we find that $\{u\in \O_{k,S}^*\mid (u,\sigma_r(u))\in X\cap V_r\}$ is contained in a finite number of cosets of $\O_{k,S}^{*\sigma_r\tau}$ in $\O_{k,S}^*$.

Since there are only finitely many such $X$ and finitely many $r$, we conclude that there exists a set $T$ as in the statement of the theorem such that $U\setminus T$ is finite.

We now prove that all of the elements in $T$ satisfy \eqref{rat} for some choice of $S'$, completing the proof of the theorem.  Let $f_1,\ldots, f_N\in L[x]$ be the monic irreducible factors of $f(x)$ over $L$.  First, consider cosets in $\O_{k,S}^*$ of the form $u_i\O_{k,S}^{*\sigma_r}$.  From a slight modification of the first part of the proof above, we need only consider cosets $u_i\O_{k,S}^{*\sigma_r}$ such that for some $j\in \{1,\ldots, N\}$ and $\epsilon>0$, there are infinitely many $u\in \O_{k,S}^{*\sigma_r}$ such that
\begin{equation}
\label{heqn}
\Sigma_{v\in M_L} \log^-\max\{|f_j(u_iu)|_v,|f_j^{\sigma_r}(\sigma_r(u_iu))|_v\} \leq -\epsilon h(u_iu).
\end{equation}
Note that $\sigma_r(u_iu)=\sigma_r(u_i)u$, since $u\in \O_{k,S}^{*\sigma_r}$.  If $f_j(u_ix)$ and $f_j^{\sigma_r}(\sigma_r(u_i)x)$ are relatively prime in $L[x]$, then the left-hand side of \eqref{heqn} is bounded from below, independent of $u\in \O_{k,S}^{*\sigma_r}$.  Since there are only finitely many $u\in \O_{k,S}^{*\sigma_r}$ with $h(u_iu)$ bounded, this contradicts the inequality \eqref{heqn} for all but finitely many $u\in \O_{k,S}^{*\sigma_r}$.  So $f_j(u_ix)$ and $f_j^{\sigma_r}(\sigma_r(u_i)x)$ have a nontrivial common factor.  Since $f_j(u_ix)$ and $f_j^{\sigma_r}(\sigma_r(u_i)x)$ are both irreducible over $L$, they must then be equal up to multiplication by a constant factor.  Thus, 
\begin{equation*}
\frac{f_j(u_ix)}{u_i^d}=\frac{f_j^{\sigma_r}(\sigma_r(u_i)x)}{\sigma_r(u_i)^d}, 
\end{equation*}
where $d=\deg f_j$.  It follows that for all $u$ in $\O_{k,S}^{*\sigma_r}$, 
\begin{equation*}
\frac{f_j(u_iu)}{u_i^d}=\sigma_r\left(\frac{f_j(u_iu)}{u_i^d}\right).
\end{equation*}
So $\frac{f_j(u_iu)}{u_i^d}\in k'$, the fixed field of $\sigma_r$.  Then for all $u\in \O_{k,S}^{*\sigma_r}$, $\frac{f_j(u_iu)}{u_i^d}$ lies in a proper subfield of $L$ not containing $k$.  So in this case \eqref{rat} holds with $S'=S_L$.

Now consider a coset of the form $u_i\O_{k,S}^{*\sigma_r\tau}$.  Again, we may assume that for some $j$ and some $\epsilon>0$, \eqref{heqn} is satisfied for infinitely many $u\in \O_{k,S}^{*\sigma_r\tau}$.  By definition, for $u\in \O_{k,S}^{*\sigma_r\tau}$ we have $\sigma_r(u)=u^{-1}$.  Let $d=\deg f_j$.  Similar to before, if $f_j(u_ix)$ and $x^df_j^{\sigma_r}(\sigma_r(u_i)/x)$ are relatively prime in $L[x]$, then it follows that
\begin{equation*}
\Sigma_{v\in M_L}\log^-\max\{|f_j(u_iu)|_v,|f_j^{\sigma_r}(\sigma_r(u_i)/u)|_v\}
\end{equation*}
is bounded from below, independent of $u\in \O_{k,S}^{*\sigma_r\tau}$.  This again gives a contradiction with \eqref{heqn} and so $f_j(u_ix)$ and $x^df_j^{\sigma_r}(\sigma_r(u_i)/x)$ must have a nontrivial common factor over $L$.  Since $f_j$ is irreducible over $L$, the two polynomials must be equal up to multiplication by a constant.  Evaluating at any $x=u'\in \O_{k,S}^{*\sigma_r\tau}$ with $f_j(u_iu')\neq 0$, we find that we must have that
\begin{equation*}
\frac{f_j(u_ix)}{f_j(u_iu')}=\frac{x^df_j^{\sigma_r}(\sigma_r(u_i)/x)}{u'^d\sigma_r(f_j(u_iu'))}.
\end{equation*}
Since $(\O_{k,S}^{*\sigma_r\tau})^2$ has finite index in $\O_{k,S}^{*\sigma_r\tau}$, we can find finitely many elements $u_1',\ldots,u_{l'}'\in \O_{k,S}^{*\sigma_r\tau}$ with $f_j(u_iu_l')\neq 0$, $l=1,\ldots, l'$, and such that for any $u\in \O_{k,S}^{*\sigma_r\tau}$, there exists some $l\in \{1,\ldots,l'\}$ with $\frac{u}{u_l'}\in (\O_{k,S}^{*\sigma_r\tau})^2$.  Let $u\in \O_{k,S}^{*\sigma_r\tau}$ and $u_l'$ chosen as above.  Then we have the identity
\begin{equation*}
\sigma_r\left(\left(\frac{u_l'}{u}\right)^{d/2}\frac{f_j(u_iu)}{f_j(u_iu_l')}\right)=\left(\frac{u_l'}{u}\right)^{d/2}\frac{f_j(u_iu)}{f_j(u_iu_l')}
\end{equation*}
and it follows that $\left(\frac{u_l'}{u}\right)^{d/2}\frac{f_j(u_iu)}{f_j(u_iu_l')}\in k'$, the fixed field of $\sigma_r$.  We can enlarge $S_L$ to a finite set of places $S'$ of $L$ such that $f_j(u_iu_l')$ is an $S'$-unit for all choices of $i$, $j$, and $l$.  Then \eqref{rat} holds for all $u\in u_i\O_{k,S}^{*\sigma_r\tau}$.
\qed

\vspace{.1in}

In the case of a cyclic subgroup of $k^*$ the theorem takes a particularly simple form.

\begin{cor}
\label{mcor}
Let $a\in k^*$.  Let $S$ be a finite set of places of $k$ such that $a$ is an $S$-unit.  Assume that for all positive integers $m$: 
\begin{enumerate}
\item  The element $a^m$ does not lie in a proper subfield of $k$.
\item  $k$ is not a quadratic extension of a field $k'$ with $N^k_{k'}(a^m)=1$.
\end{enumerate}
Let $\epsilon>0$.  Then for all but finitely many integers $n$,
\begin{equation*}
N(J(f(a^n)))<H(a^n)^{\epsilon}.
\end{equation*}
\end{cor}
\begin{proof}
Suppose that for infinitely many $n$, $N(J(f(a^n)))\geq H(a^n)^{\epsilon}$.  Then by Theorem \ref{main}, there exists $\sigma\in \Gal(L/\Q)\setminus \Gal(k/ \Q)$ and $u\in \O_{k,S}^*$ such that for infinitely many $n$, $a^n$ lies in a coset of the form $u\O_{k,S}^{*\sigma}$ or $u\O_{k,S}^{*\sigma\tau}$.  This implies that for some $m\neq 0$, $a^m\in \O_{k,S}^{*\sigma}$ or $a^m\in \O_{k,S}^{*\sigma\tau}$.  In the first case, $a^m$ lies in the proper subfield $k\cap F$ of $k$, where $F$ is the fixed field of $\sigma$.  Suppose that $a^m\in \O_{k,S}^{*\sigma\tau}$ and that $a^m$ does not lie in a proper subfield of $k$.  Then $k=\mathbb{Q}(a^m)$.  Since $\sigma(a^m)=a^{-m}$, $\sigma$ restricts to an automorphism of $k$ over $\mathbb{Q}$.  Note that $\sigma^2(a^m)=a^m$, so $\sigma$ is an automorphism of $k$ of order $2$.  Let $k'$ be the fixed field of $\sigma$.  Then $[k:k']=2$, $\Gal(k/k')=\{\id, \sigma\}$, and $N^k_{k'}(a^m)=a^m\sigma(a^m)=1$.
\end{proof}

We give an example related to Fibonacci numbers to show the likely necessity of the less obvious condition (b) in Corollary \ref{mcor}.

\begin{exm}
Let $k=\mathbb{Q}(\sqrt{5})$ and $a=\varphi=\frac{1+\sqrt{5}}{2}\in k^*$.  Let $S$ consist of the archimedean places of $k$ and the prime lying above $5$.  Let $f(x)=x+1$.  For $n$ odd, we have
\begin{equation*}
\frac{\varphi^{2n}+1}{\varphi^n\sqrt{5}}=F_n,
\end{equation*}
where $F_n$ is the $n$th Fibonacci number.  So
\begin{equation*}
f(\varphi^{2n})\O_{k,S}=F_n\O_{k,S}.
\end{equation*}
A well-known na\"ive heuristic argument suggests that there should be infinitely many Fibonacci numbers that are prime and congruent to $\pm 2\pmod{5}$ (so that these primes are inert in $k$).  In this case, there would be an $\epsilon>0$ and infinitely many values of $n$ such that $N(J(f(\varphi^n)))=N(f(\varphi^n))>H(\varphi^n)^{\epsilon}$.  This doesn't contradict Corollary \ref{mcor} as $N^k_\mathbb{Q}(\varphi^2)=1$.
\end{exm}

We now give a slight reformulation of our results.

\begin{dfn}
Let $D$ be an effective divisor on $\mathbb{P}^1$ defined over $k$ and
supported on $\mathbb{P}^1\setminus \{0,\infty\}=\G_m$.  Let $a\in
k^*, a\not\in \Supp D$, where $\Supp D$ is the support of $D$.  Let
$h_D$ be the absolute logarithmic height associated to $D$ and let
$h_D=\sum_{v\in M_k}h_{D,v}$ be a decomposition of $h_D$ into
local heights (Weil functions).  For a place $v\in M_k$ associated to a prime $\mathfrak{p}$ lying above a prime $p\in \mathbb{Z}$, let $f_v=f_\mathfrak{p}=[\O_k/\mathfrak{p}:\mathbb{Z}/p\mathbb{Z}]$.  Set $f_v=1$ if $v|\infty$ .  We define the degree one height of $a$
with respect to $k$ and $D$ by
\begin{equation*}
h_{D,\deg_1(k)}(a)=\sum_{\substack{v\in M_k\\f_v=1}}h_{D,v}(a).
\end{equation*}
Similarly, we define
\begin{equation*}
h_{D,\deg_{>1}(k)}(a)=\sum_{\substack{v\in M_k\\f_v>1}}h_{D,v}(a).
\end{equation*}
\end{dfn}

Note that 
\begin{equation*}
h_D(a)=h_{D,\deg_1(k)}(a)+h_{D,\deg_{>1}(k)}(a)
\end{equation*}
and by standard properties of heights, $h_{D,\deg_1(k)}$ and
$h_{D,\deg_{>1}(k)}$ depend on the choice of $h_D$ and the local height
functions only up to $O(1)$.  

%Also note that since $k/\Q$ is Galois,
%for any prime $p$ of $\Z$, every prime $v$ of $\O_k$ lying over $p$
%has the same degree over the prime field.  In particular, if one such
%prime has degree one, then $p$ is totally split in $k$.

\begin{cor}\label{heightmain}
Let $D$ be an effective divisor on $\mathbb{P}^1$ defined over $k$ and supported on $\mathbb{P}^1\setminus \{0,\infty\}$.  Let $f(x)\in \O_k[x]$ be a polynomial defining $D$ with monic irreducible factors $f_1,\ldots,f_n$ over $L$.  Let $\epsilon>0$.    Then there exists a finite set of places $S'$ of $L$ such that for every $u\in \O_{k,S}^*$ either
\begin{enumerate}
\item
\label{maineqn2}
\begin{equation*}
h_{D,\deg_{>1}(k)}(u)<\epsilon h_D(u)
\end{equation*}
or\\
\item
for some $i$,
\begin{equation*}
f_i(u)\O_{L,S'}=\alpha \O_{L,S'}
\end{equation*}
for some $\alpha$ that lies in a proper subfield of $L$ not containing $k$.
\end{enumerate}
\end{cor}

All but finitely many elements not satisfying \eqref{maineqn2} are again contained in a set $T$ as in Theorem \ref{main}.  There is also a similar reformulation of Corollary \ref{mcor} in terms of $h_{D,\deg_{>1}(k)}(u)$.

We will need a lemma.

\begin{lem}
\label{Roth}
Let $D$ be as in Corollary \ref{heightmain}.  For any finite set of places $S'\subset M_k$ and any $\epsilon>0$,
\begin{equation}
\label{ineq}
\sum_{v\in S'}h_{D,v}(u)<\epsilon h(u)+O(1)
\end{equation}
for all $u\in \O_{k,S}^*$.
\end{lem}
\begin{proof}
It suffices to show this for $D$ a
point (not equal to $0$ or $\infty$) and $S'\supset S$.  Let
$E=0+\infty$.  Since $u$ is an $S'$-unit, we have 
\begin{equation*}
\sum_{v\in S'}h_{E,v}(u)=2h(u)+O(1).  
\end{equation*}
By Roth's theorem, 
\begin{equation*}
\sum_{v\in S'}h_{D+E,v}(u)=\sum_{v\in S'}h_{D,v}(u)+2h(u)+O(1)
<(2+\epsilon)h(u)+O(1),
\end{equation*}
which gives \eqref{ineq}.
\end{proof}

In particular, it follows from Lemma \ref{Roth} that Corollary \ref{heightmain}
remains true if we add finitely many local heights to
$h_{D,\deg_{>1}(k)}$ (e.g., all the archimedean ones).\\

\noindent
{\it Proof of Corollary \ref{heightmain}:} \/ We may take as local height functions associated to $D$ the
functions
\begin{equation*}
h_{D,v}(a)=\max\left\{-\log|f(a)|_{v},0\right\}, \quad v\in M_k.
\end{equation*}

Then for all $u\in \O_{k,S}^*$,
\begin{align*}
h_{D,\deg_{>1}(k)}(u)&=\sum_{\substack{v\in M_k\\f_v>1}}h_{D,v}(u)\\
&=\sum_{\substack{v\in M_k\setminus S\\f_v>1}}\max\left\{-\log|f(u)|_{v},0\right\}+\sum_{\substack{v\in S\\f_v>1}}h_{D,v}(u)\\
&=\frac{1}{[k:\mathbb{Q}]}\log N(J(f(u)))+\sum_{\substack{v\in S\\f_v>1}}h_{D,v}(u)\\
&< \epsilon h(u)+O(1)
\end{align*}
by Theorem \ref{main} and Lemma \ref{Roth}.
\qed

\section{Elliptic Curves}

Theorem~\ref{main} has a conjectural analogue for elliptic curves,
following from a conjectural analogue of Lemma~\ref{vojta}.

\begin{conj}[Vojta]\label{vojtaconj}
Let $E$ be an elliptic curve defined over a number field $k$.  Let $h$ be an ample height
function on $E$.  Let $B\subset E(\kbar)\times E(\kbar)$ be a finite set of points with $B$ defined over $k$.  Let $\pi\colon X\to E\times E$ be the morphism obtained by blowing up the points in $B$ and let $Y$ be the
exceptional divisor of $\pi$.  Let $h_Y$ be a logarithmic height
function with respect to $Y$.  Let
$\epsilon>0$.  There exists a proper Zariski closed
subset $Z(\epsilon)$ of $X$ such that for every $(P,Q)\in (E\times
E)(k)-\pi(Z(\epsilon))$, we have
\[h_Y(\pi^{-1}(P,Q))\leq \epsilon (h(P)+h(Q)) + O(1).\]
\end{conj}
Conjecture~\ref{vojtaconj} is a special case of a much more
general set of conjectures made by Vojta -- see \cite{Vo} for details.

This enables us to deduce an analogue of Theorem~\ref{main} for
elliptic curves.  As in the previous section, it will be convenient to list the notation used:\\

\begin{tabular}{ll}
$k$ & Fixed number field \\
$\ell/k$ & Fixed nontrivial extension of $k$ \\
$L$ &  Galois closure of $\ell$ over $k$ \\
$\Gal(L/k)$ & Galois group of $L$ over $k$\\
$\O_k$ & Ring of integers of $k$\\
$S$ & \parbox[t]{5in}{Fixed finite set of places of $L$ consisting of:
\begin{itemize}
\item The archimedean places of $L$
\item The places of $L$ ramified over $k$
\end{itemize}} \\
$\O_{L,S}$ & The ring of $S$-integers of $L$ \\
$E$ & Fixed elliptic curve given by a Weierstrass equation \\
& $y^2=x^3+ax+b$, $a,b\in \O_k$ \\
$E(\ell)^{\nu\sigma}$ & For $\nu \in\Aut(E)$ and $\sigma\in \Gal(L/k)$, the subgroup of \\ 
& points $x\in E(\ell)$ satisfying $\nu\sigma(x)=x$\\
$D$ & Fixed effective and non-trivial $\ell$-rational divisor on $E$ \\
$D_1,\ldots, D_N$ & The irreducible components of $D$ over $L$\\
$I_D(P)$ & Ideal associated to $D$ and $P$ (see Definition \ref{defidp}) \\ 
$J_D(P)$ & \parbox[t]{5in}{The smallest divisor ideal of $I_D(P)$ supported on
primes \\ $\mathfrak{p}$ of $\O_{\ell}$ with $[\O_\ell/\mathfrak{p}:
\O_k/(\O_k\cap\mathfrak{p})]>1$}\\
$N(I)$ & Absolute norm of an ideal $I$ of $\O_\ell$\\
$H_D(P)$ & Multiplicative height function on $E$ corresponding to $D$ \\
$h_D(P)$ & Logarithm of $H(P)$: $h_D(P)=\log H_D(P)$.
\end{tabular}

\vspace{.1in}

We will also need the following definitions.
\begin{dfn}
\label{defidp}
Let $E:y^2=x^3+ax+b$, $a,b\in \O_k$, be an elliptic curve.  Let $L$ be a number field containing $k$ and let $P, Q\in E(L)$, $P\neq Q$.  Let $P-Q=(x_0,y_0)\in E(L)$.  Define 
\begin{equation*}
I_Q(P)=\prod_{\mathfrak{p}\subset \O_L} \mathfrak{p}^{\max\{-\frac{1}{2}\ord_{\mathfrak{p}}x_0,0\}},
\end{equation*}
where $\mathfrak{p}$ runs over all (finite) primes of $\O_L$ (this is well-defined, independent of $L$, if we identify ideals $\mathfrak{a}\subset \O_L$ and $\mathfrak{a}\O_{L'}$, when $L\subset L'$).  If $D=\sum_{i=1}^nQ_i$, $Q_i\in E(\kbar)$, is a nontrivial effective divisor on $E$, then we define
\begin{equation*}
I_D(P)=\prod_{i=1}^n I_{Q_i}(P).
\end{equation*}
\end{dfn}

\begin{dfn}
Let $P\in E(\ell), P\not\in\mbox{Supp}(D)$.  We define the height of
$P$ with respect to degree one primes of $\ell/k$ by
\begin{equation*}
h_{D,\deg_1(\ell/k)}(P)=\sum_{\substack{v\in M_k}}\sum_{\substack{w\in M_\ell\\w|v\\ f_{w/v}=1}}h_{D,w}(P),
\end{equation*}
where $h_{D,w}$ denotes a local Weil height with respect to $D$ and
$w$ and $f_{w/v}$ is the inertia degree of $w$ over $v$.  Similarly, define
\begin{equation*}
h_{D,\deg_{>1}(\ell/k)}(P)=\sum_{\substack{v\in M_k}}\sum_{\substack{w\in M_\ell\\w|v\\ f_{w/v}>1}}h_{D,w}(P).
\end{equation*}
\end{dfn}

Note that, as in the previous section, we have
\begin{equation*}
h_{D,\deg_1(\ell/k)}(P)+h_{D,\deg_{>1}(\ell/k)}(P)=h_D(P)+O(1).
\end{equation*}

For $P\in E(\ell)$ and $D$ a divisor on $E$ defined over $\ell$, the norm $N(I_D(P))$ is essentially just the nonarchimedean part of the (relative) height $H_{D,\ell}(P)=H_{D}(P)^{[\ell:\mathbb{Q}]}$ and $\log N(J_D(P))=[\ell:\Q]h_{D,\deg_{>1}(\ell/k)}(P)$ (up to $O(1)$).  We will assume the local heights are chosen so that this last statement is an equality.

We can now state the following theorem, which in the simplest case where $\ell/k$ is Galois, says, roughly, that the
height of $P$ with respect to $D$ is ``mostly'' supported on the degree one primes of
$\ell/k$, unless the ideal $I_D(P)$ is coming from a proper subfield of $\ell$.

\begin{thm}\label{ellmain}
Let $\epsilon>0$.  Assume that Conjecture~\ref{vojtaconj} holds.  Then for every $P\in E(\ell)$,  either
\begin{enumerate}
\item 
\label{elleqn}
\begin{equation*}
\frac{1}{[\ell:\Q]}\log N(J_D(P))=h_{D,\deg_{>1}(\ell/k)}(P)<\epsilon h_D(P),
\end{equation*}
or\\
\item 
\label{elleqnb}
for some $i$,
\begin{equation*}
I_{D_i}(P)\O_{L,S}=\mathfrak{a}\O_{L,S}
\end{equation*}
for some ideal $\mathfrak{a}\subset \O_{k'}$, where $k'$ is a proper subfield of $L$ not containing $k$ (in particular, if $\ell/k$ is Galois, $\mathfrak{a}$ is contained in a proper subfield of $\ell$).
\end{enumerate}
The set of points in $E(\ell)$ not satisfying \eqref{elleqn} is contained in a finite union of cosets in $E(\ell)$ of the form
\begin{equation*}
T=\cup_{i=1}^m P_i+E(\ell)^{\nu_i\sigma_i},
\end{equation*}
where $P_i\in E(\ell)$, $\sigma_i\in \Gal(L/k)\setminus \Gal(L/\ell)$, and $\nu_i\in \Aut(E)$ for $i=1,\ldots, m$.
\end{thm}

\noindent
{\it Proof:} \/ 
Let $D_{\rm red}$ be the reduced divisor associated to $D$.  Then for some positive integer $c$, $D<cD_{\rm red}$ and we have $h_D<ch_{D_{\rm red}}+O(1)$ and $h_{D,\deg_{>1}(\ell/k)}<ch_{D_{\rm red},\deg_{>1}(\ell/k)}+O(1)$.  So without loss of generality we may assume that $D$ is a reduced divisor.
Let
\begin{equation*}
U=\{P\in E(\ell)\mid h_{D,\deg_{>1}(\ell/k)}(P)\geq \epsilon h_D(P)\}.
\end{equation*}

Let $L$ be a Galois closure of $\ell/k$.  Let $w'\in M_L$ lie above $w\in M_\ell$ and $v\in M_k$.  As in the proof of Theorem \ref{main}, if $f_{w/v}>1$, then there exists $\sigma\in \Gal(L/k)\setminus \Gal(L/\ell)$ such that $\sigma(w')=w'$.  Let $\Gal(L/k)\setminus \Gal(L/\ell)=\{\sigma_1,\ldots,\sigma_m\}$.  For $i=1,\ldots, m$, let
\begin{equation*}
h_{D,\deg_{>1}(L/k)}^{(i)}(P)=\sum_{\substack{v\in M_k}}\sum_{\substack{w\in M_L\\w|v\\ f_{w/v}>1\\\sigma_i(w)=w}}h_{D,w}(P).
\end{equation*}
Then
\begin{equation*}
h_{D,\deg_{>1}(\ell/k)}(P)\leq \sum_{i=1}^m h_{D,\deg_{>1}(L/k)}^{(i)}(P).
\end{equation*}
Let
\begin{equation*}
U_i=\left\{P\in U\mid h_{D,\deg_{>1}(L/k)}^{(i)}(P)\geq \frac{\epsilon}{m} h_D(P)\right\}.  
\end{equation*}
Then $U\subset\cup_{i=1}^m U_i$.  Let $r\in \{1,\ldots, m\}$.  If $w\in M_L$ and $\sigma_r(w)=w$, then $h_{D,w}(P)=h_{\sigma_r(D),w}(\sigma_r(P))$ and so
\begin{equation*}
\min \{h_{D,w}(P),h_{\sigma_r(D),w}(\sigma_r(P))\}=h_{D,w}(P).
\end{equation*}
Let $\pi\colon X\to E\times E$ be the morphism obtained by blowing up the points in $D\times \sigma_r(D)\subset E\times E$ and let $Y$ be the
exceptional divisor of $\pi$.  By well-known properties of heights, for $(P,Q)\not\in D\times \sigma_r(D)$ and $w\in M_L$, we can choose
\begin{equation*}
h_{Y,w}(\pi^{-1}(P,Q))=\min\{h_{D,w}(P),h_{\sigma_r(D),w}(Q)\}.
\end{equation*}

Let $V_r=\{(P,\sigma_r(P))\mid P\in U_r\}$.  It follows that for $(P,\sigma_r(P))\in V_r$, we have
\begin{align*}
h_Y(\pi^{-1}(P,\sigma_r(P)))&\geq h_{D,\deg_{>1}(L/k)}^{(r)}(P)\geq \frac{\epsilon}{m} h_D(P)\\
&>\frac{\epsilon}{2m}(h(P)+h(\sigma_r(P))+O(1).
\end{align*}
Then by Conjecture \ref{vojtaconj} $V_r$ is contained in a proper Zariski closed subset of $E\times E$.  Let $C$ be a positive-dimensional component of the Zariski closure of $V_r$.  Then $C$ is a curve with
infinitely many rational points on it.  By Faltings' theorem, $C$ is a translate of a one-dimensional abelian subvariety $E'$ of $E\times E$.

Any irreducible one-dimensional abelian subvariety  of $E\times E$ must
be an elliptic curve isogenous to $E$, via projection onto $E$.   Since $E'$ is clearly
not a fibre of either of the two projection maps, there are
two isogenies $\phi,\psi\colon E'\to E$ induced by the two projection
maps, with dual isogenies $\hat{\phi}$ and $\hat{\psi}$ from $E$ to
$E'$.  If $R=(P,Q)\in E'\subset E\times E$, then $\hat{\phi}\phi(R)=\hat{\phi}(P)=(\deg \hat{\phi})R$ and similarly $\hat{\psi}(Q)=(\deg \hat{\psi})R$.  Thus, $E'$ is contained in the set $\{(P,Q)\in E\times E\mid
(\deg \hat{\psi})\hat{\phi}(P)=(\deg \hat{\phi})\hat{\psi}(Q)\}$.  Composing with an isogeny to $E$ we find that there are nonzero endomorphisms
$f$ and $g$ of $E$ such that $E'\subset \{(P,Q)\in E\times E\mid f(P)=g(Q)\}$.  Note that if $(P_0,\sigma_r(P_0)),(P,\sigma_r(P))\in V_r\cap C$ then $(P-P_0,\sigma_r(P-P_0))\in E'$.  It follows that there are points of the form $(P,\sigma_r(P))\in E'$ with $P\in E(\ell)$ such that $f(P)=g(\sigma_r(P))$. 

Let $K=\mbox{End}(E)\otimes\Q$.  Then $\sigma_r$ is an element of a
finite group acting on the finite-dimensional $K$-vector space
$V=E(L){\otimes}_{\mbox{\tiny End(E)}} K$.  Thus, the eigenvalues of the action
of $\sigma_r$ must be roots of unity.  But from the above, $f/g$ is an eigenvalue of
$\sigma_r$.  So we deduce that
$f/g\in K$ is a root of unity.  Since $K$ is contained in a quadratic
extension of $\Q$, this means that $f/g\in\{\pm 1,\pm
i,\pm\gamma,\pm\gamma^2\}$, where $\gamma$ denotes a primitive sixth
root of unity.  Write $g=\nu f$.  Composing both sides with the dual endomorphism to $f$, we may assume that $f=m$, where $m$ is a positive integer. Then for $(P,\sigma_r(P)),(P_0,\sigma_r(P_0))\in V_r\cap C$, we have $m(P-P_0)=\nu\sigma_r(m(P-P_0))$.  This implies that $U_r$ is contained in finitely many cosets of the form $P_i+E(\ell)^{\nu_i\sigma_r}$ in $E(\ell)$, where $P_i\in E(\ell)$ and $\nu_i\in \Aut(E)$.  So the set of points in $E(\ell)$ not satisfying \eqref{elleqn} is contained in a set $T$ as in the theorem.

We now show that the set of points in the set $T$ not satisfying condition (a) satisfies condition \eqref{elleqnb}.  Let $D_1,\ldots, D_N$ be the irreducible components of $D$ over $L$.  Consider a coset in $E(\ell)$ of the form $P_r+E(\ell)^{\nu_r\sigma_r}$, $P_r\in E(\ell), \nu_r\in \Aut(E), \sigma_r\in \Gal(L/k)\setminus \Gal(L/\ell)$.  From the first part of the proof, we need only consider cosets such that for some $i$, some $\epsilon>0$, and infinitely many elements $P\in E(\ell)^{\nu_r\sigma_r}$, we have 
\begin{equation*}
\sum_{w\in M_L}\min\{h_{D_i,w}(P+P_r),h_{\sigma_r(D_i),w}(\sigma_r(P+P_r))\}> \epsilon h(P).
\end{equation*}
Let $\phi:E\to E$ be the morphism $\phi(P)=\nu_r^{-1}P+\sigma_r(P_r)$.  Since $\sigma_r(P+P_r)=\nu_r^{-1}P+\sigma_r(P_r)$ for $P\in E(\ell)^{\nu_r\sigma_r}$, we have (up to $O(1)$)
\begin{equation*}
h_{\sigma_r(D_i),w}(\sigma_r(P+P_r))=h_{\sigma_r(D_i),w}(\phi(P))=h_{\phi^*\sigma_r(D_i),w}(P).
\end{equation*}
Let $\tau$ be translation by $P_r$.  So $h_{D_i,w}(P+P_r)=h_{\tau^*D_i,w}(P)+O(1)$.  So for infinitely many $P\in E(\ell)$,
\begin{equation}\label{exceptions}
\sum_{w\in M_L}\min\{h_{\tau^*D_i,w}(P),h_{\phi^*\sigma_r(D_i),w}(P)\}> \epsilon h(P).
\end{equation}

If $\tau^*D_i$ and $\phi^*\sigma_r(D_i)$ have empty intersection, then as is well known, $\sum_{w\in M_L}\min\{h_{\tau^*D_i,w}(P),h_{\phi^*\sigma_r(D_i),w}(P)\}$ is bounded independent of $P$, contradicting inequality \eqref{exceptions}.  So $\tau^*D_i\cap \phi^*\sigma_r(D_i)\neq \emptyset$.  Since $D_i$ is irreducible over $L$, this implies that $\tau^*D_i=\phi^*\sigma_r(D_i)$.  

It follows from the definition that for any translation $\tau_0$ and any automorphism $\nu\in \Aut(E)$, $I_{D}(\tau_0(P))=I_{\tau_0^*D}(P)$ and $I_{D}(\nu P)=I_{\nu^*D}(P)$.  This implies that for all $P\in E(\ell)^{\nu_r\sigma_r}$,
\begin{align*}
\sigma_r (I_{D_i}(P+P_r))&=I_{\sigma_r(D_i)}(\sigma_r(P)+\sigma_r(P_r))=I_{\sigma_r(D_i)}(\phi(P))=I_{\phi^*\sigma_r(D_i)}(P)\\
&=I_{\tau^*D_i}(P)\\
&=I_{D_i}(P+P_r).
\end{align*}
So $\sigma_r$ fixes the ideal $I_{D_i}(P_r+P)$, $P_r+P\in P_r+E(\ell)^{\nu_r\sigma_r}$, which implies that $I_{D_i}(P+P_r)\O_{L,S}=\mathfrak{a}\O_{L,S}$ for some ideal $\mathfrak{a}$ of $\O_{k'}$, where $k'$ is the fixed field of $\sigma_r$.

%By Siegel's Theorem, for any large enough integers
%$n$, we have $\sum_{v\mid S} h_{D,v}(P) < (\epsilon/2)h_D(P) +
%O(1)$.  By hypothesis, we therefore have, for certain arbitrarily
%large integers $n$:
%\begin{align*}
%h_{D,nonsplit}(nP) - \sum_{v\mid S} h_{D,v}(nP) &> \epsilon h_D(nP) - 
%\sum_{v\mid\infty} h_{D,v}(nP) + O(1)
%\end{align*}
%This implies, using the notation of Theorem~\ref{ellmain}:
%\begin{align*}
%\log N(J(f(u))) &> (\epsilon/2) h_D(nP) + O(1)
%\end{align*}
%Corollary~\ref{heightellmain} now follows from Theorem~\ref{ellmain}.
\qed

\vspace{.1in}

If we restrict to cyclic subgroups of $E(\ell)$, we obtain the following simpler version of Theorem \ref{ellmain}.

\begin{cor}\label{ellcor}
Let $P\in E(\ell)$ and $\epsilon>0$.  If Conjecture~\ref{vojtaconj} holds, then either
\begin{equation*}
h_{D,\deg_{>1}(\ell/k)}(nP) <\epsilon h_D(nP)
\end{equation*}
for all but finitely many integers $n$, or there exists a proper
subfield $k'\subsetneqq\ell$ of $\ell$, a positive integer $m$, an
elliptic curve $E'/k'$, and an isomorphism $\phi:E\to E'$ over $\ell$
such that $\phi(mP)$ is a $k'$-rational point on $E'$.
\end{cor}

\begin{proof}
Suppose that for infinitely many $n$, $h_{D,\deg_{>1}(\ell/k)}(nP) <\epsilon h_D(nP)$.  It follows from Theorem \ref{ellmain} that for some $m>0$, $\sigma\in \Gal(L/k)\setminus \Gal(L/\ell)$, and $\nu\in \Aut(E)$, we have $mP\in E(\ell)^{\nu^{-1}\sigma}$, or $\sigma(mP)=\nu mP$.  From this it follows that $mP$ is a point on a twist of $E$, defined over $k'\cap \ell$, where $k'$ is the fixed field of $\sigma$.
\end{proof}

At the time of writing, Conjecture~\ref{vojtaconj} is known only in
the following special case.  See \cite{McK} for a proof, and \cite{Si}
for a discussion of the implications of Vojta's Conjecture in this
context.

\begin{thm}[McKinnon]
\label{Davidthm}
Let $E$ be an elliptic curve over a number field $\ell$.  Let
$R=\End_\ell(E)$.  Let $M$ be a cyclic $R$-submodule of $E(\ell)$.
Then Conjecture~\ref{vojtaconj} holds for $(P,Q)\in M\times M\subset
(E\times E)(\ell)$; i.e., in the notation of
Conjecture~\ref{vojtaconj}, there exists a proper Zariski closed
subset $Z(\epsilon)$ of $X$ such that for every $(P,Q)\in M\times
M-\pi(Z(\epsilon))$, we have
\[h_Y(\pi^{-1}(P,Q))\leq \epsilon (h(P)+h(Q)) + O(1),\]
where $h_Y$ is a logarithmic height function associated to the
exceptional divisor on the blowup $X$ of $E\times E$ at a finite set of points and $h$
is any fixed ample logarithmic height on $E$.
\end{thm}

\begin{thm}\label{ellmain2}
Let $E$ be an elliptic curve over a number field $k$ with complex
multiplication.  Let $\ell$ be the compositum of $k$ with the
imaginary quadratic field $\End(E)\otimes \mathbb{Q}$.  Let $D$ be a nontrivial effective divisor on $E$ defined over $\ell$.  Let $P\in E(\ell)$ and $\epsilon>0$.  Then either 
\begin{equation*}
h_{D,\deg_{>1}(\ell/k)}(nP) <\epsilon h_D(nP)
\end{equation*}
for all but finitely many $n>0$, or there exists a positive integer $m$, an elliptic curve $E'/k$,
and an isomorphism $\phi:E\to E'$ over $\ell$ such that $\phi(mP)$ is
a $k$-rational point on $E'$.
\end{thm}

\noindent
{\it Proof:} \/ If $\ell=k$ then the theorem is vacuous.  So suppose
that $\ell$ is a quadratic extension of $k$.  Let $R=\End(E)$.  First,
we note that $R[E(k)]$ has finite index in $E(\ell)$.  Indeed, as is
well-known \cite[Exerc.~ X-10.16]{Si2}, we have $\rk E(\ell)=\rk
E(k)+\rk E'(k)$, where $E'$ is a quadratic twist of $E$ over $\ell$.
If $\ell=k(\sqrt{N})$, $N\in \mathbb{\Z}$, then any element
$n\sqrt{N}\in R$, with $n$ a positive integer, induces an isogeny (over
$k$) between $E$ and a quadratic twist $E'$ of $E$ over $\ell$.  Thus,
$\rk E(k)=\rk E'(k)$ and we have $\rk E(\ell)=2\rk E(k)=\rk R[E(k)]$.

Next, we claim that Theorem \ref{Davidthm} actually holds under the slightly weaker assumption that $M$ contains a cyclic $R$-submodule $M'$ of finite index $m$.  Indeed, one easily reduces to considering the case where $X$ is the blow-up of $E\times E$ at the origin $(\O,\O)$ and $Y$ is the exceptional divisor.  The claim then follows by applying Theorem \ref{Davidthm} to $M'$ and from the facts $h_Y(\pi^{-1}(P,Q))\leq h_Y(\pi^{-1}(mP,mQ))+O(1)$, $(P,Q)\neq (\O,\O)$, and $h(mP)=m^2h(P)+O(1)$.

Let $m$ be the index of $R[E(k)]$ in $E(\ell)$.  Let $P\in E(\ell)$.
Then we have $mP=\phi(Q)$, for some $Q\in E(k)$ and some $\phi\in R$.  Let $\sigma$
be the unique non-identity element of $\Gal(\ell/k)$.  Then
$m\sigma(P)=\sigma(mP)=\sigma(\phi(Q))=(\sigma \phi)(Q)$, so $mP$ and $m\sigma(P)$ both
belong to the cyclic $R$-submodule $RQ$ of $E(\ell)$ generated by $Q$.  So $RQ$ has finite index in the subgroup of $E(\ell)$ generated by $RQ$, $P$, and $\sigma(P)$.  Then by our earlier claim, Conjecture~\ref{vojtaconj} holds for the
points $(nP,n\sigma(P))\in (E\times E)(\ell)$, $n\in \mathbb{Z}$.
But now the same proof as in Theorem~\ref{ellmain} and Corollary \ref{ellcor} works, completing
the proof.  \qed

\end{document}